\documentclass[12pt,reqno]{amsart}
\usepackage{amssymb}
\usepackage{epsf,graphicx}
 
\newtheorem{theorem}{Theorem}[section]
\newtheorem{proposition}[theorem]{Proposition}
\newtheorem{lemma}[theorem]{Lemma}

\numberwithin{equation}{section}

\newcommand{\Z}{\mathbb{Z}}

\newcommand{\X}{\mathbb{X}}

\newcommand{\gn}{\mathfrak{n}}

\newcommand{\C}{\mathbb{C}}
\newcommand{\bP}{\mathbb{P}}
\newcommand{\SMS}{\mathcal{M}_{SL}}
\newcommand{\SMG}{\mathcal{M}_{GL}}
\newcommand{\SO}{\mathcal{O}}
\newcommand{\SI}{\mathcal{I}}
\newcommand{\SSX}{\mathcal{X}}
\newcommand{\SZ}{\mathcal{Z}}

\begin{document}
\baselineskip=15pt

\title[Framed Parabolic Sheaves on a Trinion]{Framed Parabolic Sheaves on a Trinion}

\author[I. Biswas]{Indranil Biswas}

\address{School of Mathematics, Tata Institute of Fundamental
Research, 1 Homi Bhabha Road, Mumbai 400005, India}

\email{indranil@math.tifr.res.in}

 \author[J. Hurtubise]{Jacques Hurtubise}

\address{Department of Mathematics, McGill University, Burnside
Hall, 805 Sherbrooke St. W., Montreal, Que. H3A 2K6, Canada}

\email{jacques.hurtubise@mcgill.ca}

\subjclass[2000]{14P99, 53C07, 32Q15}

\keywords{Stable vector bundles, algebraic curves, framed parabolic structures, trinion}

\date{}

\begin{abstract} 
We consider for structure groups ${\rm SU}(n)\,\subset\, {\rm SL}(n,\C)$ a densely defined toric structure on the moduli of framed parabolic sheaves on a 
three-punctured sphere, which degenerates to an actual toric structure. In combination with previous degeneration results, these extend 
to similar moduli for arbitrary Riemann surfaces.
\end{abstract}

\maketitle

\tableofcontents
\section{Introduction}

In 1992, Jeffrey and Weitsman, \cite{JW}, showed how the space of ${\rm SU}(2)$ representations of the 
fundamental group of a Riemann surface possessed an integrable Hamiltonian system with continuous 
Hamiltonians, smooth away from a small locus, and giving  a torus action on an open dense set; in a sense it is an 
``almost toric'' structure on the space. This structure was associated to the Goldman flows on 
representations associated to a trinion (three-punctured spheres) decomposition of the surface (it is also
called ``pant decomposition''). 
The construction allowed them to give an interpretation of the Verlinde formulae, \cite{Ve}, in 
terms of Bohr-Sommerfeld quantization, by counting lattice points in the moment polytope.
 
We recall that via the theorem of Narasimhan and Seshadri, \cite{NS}, the space of representations of 
the fundamental group in ${\rm SU}(n)$ (more generally, in a compact semi-simple group $G$) has another 
interpretation as a space of stable holomorphic rank $n$ vector bundles equipped with a global 
volume form (more generally, stable principal $G_\C$-bundles, where $G_\C$ is
the complexificaton of $G$). In the holomorphic interpretation of the 
moduli space as stable principal ${\rm SL}(2,\C)$ bundles the Verlinde formulae give the dimensions 
of spaces of sections of tensor powers of a natural determinant line bundle over the moduli space, and if the variety were 
actually toric, the point count would naturally correspond to this dimension, by the general 
theory of toric varieties (e.g. as in Fulton \cite{Fu}). It is thus as if the actual moduli space were a deformation of another moduli space that is 
actually toric. 
We now have the good fortune of being able to refer to a beautiful series of results on the structure of these toric degenerations, 
valid in large generality, both in the algebraic framework (see the papers of Anderson \cite{An}, Kaveh and Khovansky \cite{KK})) and in the 
symplectic framework (see the paper of Harada and Kaveh \cite{HK}). 

Indeed, in our case, a suitable toric candidate for ${\rm SU}(2)$ was produced in \cite{HJS1}, \cite{HJS2}. In these 
papers, a moduli space of ``framed parabolic sheaves'' was constructed over punctured Riemann 
surfaces, with the standard parabolic structures of bundles at the puncture (i.e., flags) being 
enriched by an additional framing variable (a trivialization of the top exterior powers of the 
subquotients of the flag). This space also has a symplectic interpretation as an ``imploded 
cross-section'' of the space of all framed representations of the fundamental group of the punctured 
surface, these representations being thought of as flat connections on a bundle, and the framing 
being the additional data of a trivialization at each puncture. The space of framed parabolic 
sheaves is a ``master space'' for parabolic bundles: it gives the standard parabolic moduli, 
either symplectically by a torus reduction, or holomorphically by the corresponding
geometric invariant theoretic (GIT) quotient 
by a complex torus; in both cases the tori act on the framings.

Using the actions of tori at the punctures, we can perform a glueing procedure: we first build a nodal curve from a 
(possibly disconnected) punctured curve by joining some or all of the punctures in pairs. Above this nodal curve, for each node, we 
can build ``partially glued" sheaves from framed parabolic sheaves by identifying the top exterior powers of the subquotients of the 
flags at the two punctures by their trivializations. This actually amounts to a diagonal symplectic quotient in the symplectic picture or a 
diagonal geometric invariant theoretic quotient in the holomorphic picture by the appropriate torus actions. (In fact, with the appropriate normalizations, the quotients are by an ``anti-diagonal" torus, because the order of the eigenvalues at one puncture is the opposite of that at the other puncture.) 
 
Of course, we need to see that this space is the actual result of a degeneration. This was carried out in \cite{BHu}, where the glued space 
was obtained from the standard moduli on smooth curves as a limit associated to a nodal degeneration of the curve into a ``balloon 
animal'' made up of trinions. That paper considers both a symplectic and a holomorphic version. It may be mentioned that
the degenerations obtained are different from those considered by Bhosle in \cite{Bh2}.
 
In the case of ${\rm SL}(2,\C)$ it is already shown, \cite{HJ}, that the trinion building block is $\bP^3(\C)$, and so the glued space over the nodal curve corresponding to the glued 
trinions is indeed toric, being obtained by symplectic (or GIT) torus quotients from a product of 
the trinion spaces ($\bP^3(\C)$) : the toric variety corresponding to ${\rm SU}(2)$ representations for 
a Riemann surface is a moduli space of ``partially glued" ${\rm SL}(2,\C)$ sheaves over a degeneration 
of the Riemann surface into a nodal curve whose each irreducible component is a trinion.
 
The construction of Jeffrey and Weitsman does not work in the ${\rm SL}(n,\C)$, $n\,>\,2$, case. The reason is that they use the 
Goldman flows associated to the trinion decomposition in their construction, and there are simply
not enough of the flows. While in this more general case also we can obtain the limiting space  for 
${\rm SL}(n,\C)$ as a moduli space associated to the nodal curve, and we find that it is a torus 
reduction of the trinion spaces, the latter show no sign of being toric. The subject of this paper
is to produce a further toric degeneration of the trinion spaces, giving us then also a toric 
degeneration of the moduli spaces.

It is our pleasure to dedicate this paper to our friend and collaborator Oscar Garc\'{\i}a-Prada, a constant source of ideas and projects, 
on the occasion of his sixtieth birthday.

\section{The moduli of framed parabolic sheaves on a trinion}

\subsection {Holomorphic and symplectic moduli}

From the complex geometric point of view, a trinion, which will be denoted by $X$, is a Riemann sphere with three marked
distinct points $p_1,\, p_2,\, p_3$.

Our aim is to see how to degenerate the moduli space of ${\rm SL}(n,\C)$-framed parabolic sheaves on a 
trinion $X$. The theory of semistable 
framed parabolic ${\rm SL}(n,\C)$ sheaves, which is given in \cite{HJS2}, is
briefly recalled below.

We first consider the case when the underlying sheaf is actually locally free.
On the trinion $X$, the data are:
\begin{itemize}
\item A rank $n$ holomorphic vector bundle $E$ on $X$ of degree $0$, with a fixed isomorphism
\begin{equation}\label{e1}
\Phi\,:\, \bigwedge\nolimits^n E\, \stackrel{\sim}{\longrightarrow}\,
\SO_X\end{equation}
(the ${\rm SL}(n,\C)$ structure).

\item Elements $$\beta_j^i \,\in\,\bigwedge^{i}E^*_{p_j},\ \ j\,=\,1,\,2,\,3,\ \
i\,=\,1,\,\cdots,\,n .$$ They may vanish, but must be decomposable (so the annihilator of any
nonzero $\beta_j^i$ is a subspace $E_j^{n-i}\,\subset\, E_{p_j}$ of dimension $n-i$),
and are compatible in the sense that if $\beta_j^i ,\, \beta_j^{i'}$ are both
nonzero, with $i\,<\, i'$, then there is an element $$\gamma^{i,i'}
\, \in\, \bigwedge\nolimits^{i'-i} E^*_{p_j}$$ such that $\beta_j^{i'} \,=\, \gamma^{i,i'}\bigwedge\beta_j^{i}$.
Furthermore, $\beta_j^n\,=\, \Phi(p_j)^{-1}(1)$, for $j\,=\,1,\,2,\,3$, where $\Phi$
is the fixed isomorphism in \eqref{e1}.
\end{itemize}
The relation to parabolic bundles is simply that each non-zero $\beta^i_j$ gives
the above $n-i$ 
dimensional subspace $E^{n-i}_j $ of the fiber $E_{p_j}$, defining a
flag at each $p_j$. The fact that some of the $\beta^i_j$ can be zero allows for
non-complete flags. Note that the above ``quotients'' $\gamma^{i,i'}_j$ define volume
forms on the quotients $E_j^{n-i}/E_j^{n-i'}$, hence the term ``framed''.

These framed bundles, even if when semi-stable, can degenerate to acquire torsion. Fortunately not 
much torsion is necessary, and in fact all of the properties of the torsion described below
follow from the semistability condition and being a closed orbit in the moduli space; see Lemma 3.4 of \cite {HJS2}. (In fact, in
terms of eventual parabolic weights, the torsion only appears when the gap $\alpha_1-\alpha_n$ is $1$.) The 
set of structures extends to allow:
 
\begin{itemize}
\item A rank $n$ sheaf $E$ on $X$, locally free away from the marked points, and of degree $0$.
The torsion sheaf is supported on the reduced marked points $p_1,\, p_2,\, p_3$.
Therefore, the torsion at $p_j$ is of the form $\SO_{p_j}^{\oplus s_j}$.

\item An ${\rm SL}(n,\C)$ structure on $E$, meaning an isomorphism $\det E\, \stackrel{\sim}{\longrightarrow}\,
{\mathcal O}_X$ (see \cite[Ch.~V, \S~6]{Ko} for determinant bundle of coherent sheaves). This is equivalent to
giving an isomorphism
$$\Phi\,:\, \bigwedge\nolimits^n(E/{\rm torsion})\, \stackrel{\sim}{\longrightarrow}
\,\SO(-s_1p_1-s_2p_2-s_3p_3),$$ where $s_j$ are the above integers.

\item Elements $\beta_j^i \,\in\, \bigwedge^{i}E^*_{p_j}$, $j\,=\,1,\,2,\,3$ and
$i\,=\,1,\,\cdots,\,n-1,\,n$. They may 
vanish, but must be decomposable (so that the annihilator
of any nonzero $\beta_j^i$ is a subspace $E_j^{n-i} \,\subset\, 
E_{p_j}$ of dimension $n-i $) and if $\beta_j^i ,\, \beta_j^{i'}$ 
are both nonzero, with $i\,<\,i'$, there is
\begin{equation}\label{eg}
\gamma^{i,i'}_j\,\in\, \bigwedge\nolimits ^{i'-i} E^*_{p_j}
\end{equation}
such that $\beta_j^{i'} \,=\, \gamma^{i,i'}_j\wedge\beta_j^{i}$. If the 
rank $s_j$, of the torsion sheaf at $p_j$, is positive, then
$\beta_j^n \,=\, 0$, $\beta_ j^i \,=\, 0$ for $i\,<\,s_j$ and 
$\beta_j^{s_j}$ is a nonzero product of elements $e_j^1,\, e_j^2,\, \cdots ,\, e_j^{s_j}$
generating the dual of torsion.
\end{itemize}

In \cite{HJS2} a moduli space $\SMS\,=\, \SMS(X)$ of these framed sheaves is constructed. It also has a 
symplectic version, defined using imploded cross-sections, expressed in terms of representations 
and framings of the eigenspaces of the monodromies at the punctures \cite{HJS1}. For our case 
this amounts to the following:
\begin{itemize}
\item A representation of the fundamental group of the trinion into ${\rm SU}(n)$.

\item Trivializations of the top exterior powers of the generalized eigenspaces associated to each eigenvalue of the monodromy around each 
puncture. If we lift the eigenvalues of the monodromy around the puncture $p_i$ to the fundamental alcove in the Lie algebra, i.e., take 
logarithms, these are given by $$\alpha^1(i)\,\geq\,\alpha^2(i)\,\geq\, \alpha^3(i)\,\geq \, \cdots \,\geq\, \alpha^n(i)
\,\geq \,\alpha^1(i) -1,\ \ \sum_{j=1}^n \alpha^j(i) \,=\, 0.$$ Up to a 
normalization factor, the norms of the trivializations in the holomorphic picture are given by the differences $\alpha^{j }(i)- 
\alpha^{j+1}(i)$. These $\alpha^{j+}(i)- \alpha^{j+1}(i)$ serve as the moment maps for the natural torus actions on the trivializations. Under 
the standard equivalence with parabolic bundles, the $\alpha^j(i)$ also give the weights associated to the geometric
invariant theoretic construction of the moduli space of parabolic bundles.
\end{itemize}

On the holomorphic side, the addition of torsion means that we can obtain, for example, moduli spaces of bundles of different degrees 
from the same space of framed sheaves; the difference of degrees is ``made up'' by torsion. This occurs when the equality 
$\alpha^1(i)-\alpha^n(i) \,=\, 1$ holds for eigenvalues (recall that the $\sum_{j=1}^n \alpha^j(i)\,=\,0$). The point here is that while we can have 
locally free bundles with flags for $\alpha^1(i)-\alpha^n(i) \,=\, 1$, the orbits of these are not closed in the GIT construction, and we can 
degenerate so that the bundle acquires nontrivial torsion; these will represent the closed orbits in the GIT construction.

The moduli space of these framed sheaves is represented by a standard quotient construction: twist the sheaf $E$ by ${\SO}_{X}(k)$
for $k$ sufficiently large, and let $Z$ be the space of sections $H^0(X,\, E(k))\,=\,
H^0(X,\, E\otimes {\SO}_X(k))$. We have the homomorphism
$$
\bigwedge\nolimits^nZ\,=\, \bigwedge\nolimits^n H^0(X,\, E(k))\,\longrightarrow\, H^0(X,\,
\det E(k))\,=\, H^0(X,\, (\det E)\otimes \SO_X(nk)).
$$
Following Bhosle \cite{Bh}, and Hurtubise-Jeffrey-Sjamaar \cite{HJS2}, using the trivialization of
$\det E$ the above homomorphism gives a homomorphism 
\begin{equation}\Phi\,\,\in\,\, {\rm Hom}\left(
\bigwedge\nolimits^nZ,\, H^0(X,\,\SO_X(nk))\right).
\end{equation}
This $\Phi$ represents the sheaf $E(k)$, and so it represents $E$.
Indeed, we can define a subbundle $Ann(\Phi)$ of the trivial vector bundle $\SZ \,= \,Z\times X\, \longrightarrow\, X$ by
\begin{equation}\label{AnnPhi}
Ann(\Phi) \,=\, \left\{(s,\, x)\,\in\, \SZ\,\,\big\vert\,\, \Phi( s\wedge\alpha)(x) \,=\, 0\ \text{ for all }\,
\alpha\,\in\, \bigwedge\nolimits^{n-1}\SZ\right\}\, \subset\, \SZ. 
\end{equation}
This gives an exact sequence
\begin{equation}\label{ee}
0\,\longrightarrow\, Ann(\Phi)\,\stackrel{\eta}{\longrightarrow}\, \SZ \,=\, {\SO}_X^{\oplus N}
\,\longrightarrow \,E(k)/{\rm torsion}\,\longrightarrow\, 0,
\end{equation}
where $N\,=\, \dim H^0(X,\,E(k))$. We note that the global sections of $Ann(\Phi)$
correspond to the sections of $E(k)$ with values in its torsion subsheaf. Let
$$
T\,:=\, \eta(H^0(X,\, Ann(\Phi)))\, \subset\, \SZ \,=\, {\SO}_X^{\oplus N}
$$
be the trivial subsheaf of $\SZ$ generated by $H^0(X,\, Ann(\Phi))$, where $\eta$ is the map in \eqref{ee}.

For $j\,=\, 1,\, 2,\, 3$, consider the evaluation map $$Z\,=\, H^0(X,\,E(k)) \, \longrightarrow\, E(k)_{p_j}.$$ Take
the dual of the $i$-th exterior product of it
$$
\varphi_j^i\, :\, \bigwedge\nolimits^{i} E(k)^*_{p_j} \, \longrightarrow \, \bigwedge\nolimits^{i} H^0(X,\, E(k))^*\, .
$$
Let
\begin{equation}\label{ea}
\psi_j^i\,:=\, \varphi_j^i(\beta^i_j) \,\in\, \bigwedge\nolimits^{i} H^0(X,\, E(k))^*
\end{equation}
be the image, under $\varphi_j^i$, of the exterior algebra element $\beta^i_j$ in the definition of framed flags.

The planes in the flags are given by the evaluation at $p_j$ of $Ann(\psi_j^i)$, where $\psi^i_j$ are constructed in \eqref{ea}. 
The moduli space is the GIT quotient under ${\rm GL}(N, \C)$ of tuples $\Phi, \psi_j^i$. Note that these elements are monomials in 
the exterior algebra (point-wise, in the case of $\Phi$), that is they satisfy Pl\"ucker relations, and that there are compatibility 
conditions between the $\psi_j^i$ at each $p_j$. In particular, the fact that we are getting flags at $p_j$ is due to the 
compatibility of $\Phi,\, \psi_j^i$, so that $$Ann(\Phi(p_j))\,\subset\, T + Ann(\psi_j^i),$$ and $Ann(\psi_j^i)\,\subset\, 
Ann(\psi_j^{i'})$ whenever $i\,<\,i'$.

We have two lemmas from \cite{HJS2}:

\begin{lemma}[{\cite[p.~359, Lemma 3.3]{HJS2}}]
Let $\left(\Phi,\, (\psi^i_j)_{i=1,\cdots ,m, j= 1,2,3}\right)$ be a semi-stable point. Then
the following two hold:
\begin{enumerate}
\item $\Phi\,\ne\, 0$, and $Ann(\Phi)$ is a subbundle of $\mathcal Z$ rank $N-n$. 

\item The intersection $H^0(\Sigma,\, Ann(\Phi))\bigcap (\bigcap_{i} 
Ann(\psi_j^i))$ is zero for all $j\,=\, 1,\, 2,\, 3$.
\end{enumerate}
\end{lemma}

The following lemma is on the torsion part of a semi-stable sheaf.

\begin{lemma}[{\cite[pp.~359--360, Lemma 3.4]{HJS2}}]\label{stable-sheaves}\mbox{}
\begin{enumerate}
\item The torsion subsheaf of $E$ is concentrated over the points $p_1,\, p_2,\, p_3$.

\item For any $j\,=\, 1,\, 2,\, 3$, let $x$ be a local holomorphic coordinate function on $X$ around $p_j$, with $x(p_j)\,=\,
0$. The torsion subsheaf at $p_j$ is a sum of modules ${\SO}/x^{l_{\mu}}{\SO}$. Then all the $l_{\mu}$ are one.

\item The torsion submodule $T_j$ of $E$ at $p_j$ is then a sum of terms 
${\SO}/x{\SO}$, which we write as vector spaces $\C^{s_j}$.
One then has that $s_j\,<\,n$. If $e_1,\,\cdots,\, e_{s_j}$ is a basis of $T_j$, then the 
contraction of $e_1\wedge\cdots \wedge e_{s_j}$ with $\psi_j^i$ is
non-zero for at least one $i\,\ge\, s_j$.

\item Let $\widehat{e}_1,\, \widehat{e}_2,\, \cdots ,\, \widehat{e}_{s_j}$ be elements of $(\C^N)^*
\,=\, H^0(\Sigma,\, E(k))^*$ corresponding to the duals of $e_1,\, \cdots ,\, e_{s_j}$ under 
evaluation. If the orbit of $\left(\Phi, (\psi_j^i)_{i=1,\cdots ,n; j= 1,2,3}\right)$ is 
closed, then for each $p_j$ we can choose the bases $e_1,\,\cdots ,\, e_{s_j}$ and
$\widehat{e}_1,\, \widehat{e}_2,\, \cdots ,\,\widehat{e}_{s_j}$, and complete the basis of
$\C^N $ by elements $\widehat{e}_{s_j+1},\,\cdots ,\, \widehat{e}_N$, so that there exist 
constants $c_j$, possibly zero, with 
$\psi_j^i \,=\, c_j \widehat{e}_1 \wedge\widehat{e}_2\wedge\cdots\wedge\widehat{e}_i$, for 
$i\,=\, s_j,\, \cdots ,\, n$ (``The torsion vectors come first''), $\psi_j^i \,=\, 
0$ for $i\,<\, s_j$, and $\psi_j^{s_j}\,\neq\, 0$.
\end{enumerate}
\end{lemma}

We note that in consequence the rank $s_j$ of the torsion simply records for closed orbits the vanishing of $\psi_j^n$,\, $\psi_j^i$,\,
$i \,= 1,\, \cdots,\, s_j-1$.
 
The result of \cite{HJS2} tells us that there is a moduli space of these framed sheaves on any punctured Riemann surface
$\Sigma$ --- with
both a holomorphic and a symplectic realization. In the holomorphic version, there is a natural complex torus action on the framings 
for each marked point. Symplectically, there is a real torus action on unitary framings, well defined as we are imploding. Its moment 
map is simply the set of eigenvalues of the holonomy at the marked point, and takes values in the fundamental alcove of ${\rm SU}(n)$.
 
\subsection {Glueing}
 The moduli space admits a torus action at each marked point. The various moduli spaces of parabolic bundles can be obtained as torus 
quotients of $\SMS$, or rather of strata inside $\SMS$ given by the vanishing of various $\beta_j^i$: in essence, we quotient out the 
scale given by $\gamma^{i,i'}_j$ (see \eqref{eg}). Indeed, the stability theory of parabolic bundles for
${\rm SL}(n,\C)$ involves a simplex of weights
$$\Delta \,=\, \left\{(\alpha^1\,\geq\,\alpha^2\,\geq\, \cdots\, \geq\, \alpha^n\,\geq \,\alpha^1-1)\,\,
\big\vert\,\, \sum_{i=1}^n \alpha^i \,=\, 0\right\}$$
for each marked point. This simplex of weights decomposes into faces
$$\Delta^I,\,\, \, I \,=\, \{i_1\,<\,i_2\,<\,\cdots \,<\,i_k\}$$ consisting of the weights
for which $\alpha^{i}\,>\,\alpha^{i+1}$ if and only if $i\,\in \,I$. For the trinion, there are three simplices $\Delta_j$ (one
per marked point), and the faces $\Delta_j^I$ will correspond to strata in the framed moduli space for which only the
$\beta_j^{i_l},\,\, i_l\,\in\, I$, are non-zero, i.e., for which there is a partial flag at $p_j$ indexed by $I$. With that the
stability for the framed moduli is given by the induced parabolic bundle being stable for one choice of weights in the 
corresponding $\Delta_j^I$.

There is an alternative operation: given a framed sheaf on a surface $\Sigma$ with punctures at $p_1,\, p_2$ with identical vanishing 
patterns for the $\beta_1^j,\, \beta_2^j$ (more precisely, mapped to one another by the reversal of order in $\{1,\,\cdots ,\,n\}$), we can use the trivializations to identify the respective subquotients of the flags at 
$p_1,\, p_2$. This amounts to taking a anti-diagonal torus quotient on the framings at $p_1,\, p_2$; symplectically, it amounts to setting 
$\alpha_1(j) +\alpha_2(j) \,=\, 0$, and then quotienting by the real torus.

The same operation is possible if the two points $p_1$, $p_2$ are on different surfaces $\Sigma_i$. If we start with a surface 
$\Sigma_0$ built from a set of trinions by glueing at the punctures, we can obtain a moduli space associated to $\Sigma_0$ as a torus 
quotient of a ($3g-3$)-fold product of the moduli $\SMS$ associated to a trinion by iterating the above glueing construction. Thus a 
toric structure on a deformation of the space associated to $\Sigma_0$ depends on finding one for $\SMS$.

\subsection{A generic set}\label{se2.2}

We now turn to understanding an ``almost toric" structure on the moduli $\SMS$. As this is a bit easier to describe if the top forms
$\beta_i^n$ are allowed not to coincide with the ${\rm SL}(n, \C)$ structure at the marked points, we define
$$\SMG \,\,=\,\, \{ (E,\, \beta_i^j)\,\big\vert\,\, (E,\, \beta_i^1,\,\cdots ,\,\beta_i^{n-1},\, c_i\beta_i^{n})
\,\in\, \SMS\,\ \text{ for }\,\ c_i\,\in\, \C^*\}.$$
Note that the sheaf $E$ retains its ${\rm SL}(n,\C)$ structure.

There is a Zariski open subset $U\,\subset\, \SMG$ of the moduli space, consisting of points for which 
\begin{itemize}
\item the sheaf $E$ is a vector bundle, and is trivial,

\item there are full flags at the three points $p_1,\, p_2,\, p_3$, and

\item they are all mutually transverse in the global trivializations by sections of $E$. 
\end{itemize}

We can then normalize. It is easiest to think in terms of the equally trivial $E^*$, so that our framed flags $\beta^j_i$ are now 
$j$-vectors. We choose global trivializations $$e_1,\, \cdots ,\,e_n$$ of $E^*$ so that the flag at $p_1$ is the standard
one ($(E^*)_1^i\, =\, \langle e_1, \,\cdots ,\, e_i\rangle$), the flag at $p_2$ is the ``antistandard'' one
($(E^*)_2^i \,=\, \langle e_{n-i+1},\, \cdots ,\,e_n\rangle$). This form of the flags reduces the choices 
of basis  to the diagonal torus; incorporating the framing at $p_1$ so that it is the standard one then fixes the scale of $e_1, 
\,\cdots,\, e_n$, and then the framing at $p_2$ will be given by normalizing the basis $e_i$ by non-zero
constants $a_1,\,\cdots ,\, a_n$:
$$ f_i \,=\, a_i e_i$$
so that $\beta_2^i \,=\, (a_1\cdots a_{n-i})^{-1} e_1^*\wedge\cdots\wedge e_{n-i}^*$. The third framed flag at $p_3$ can then
be given a normalized basis
\begin{align}
g_1 \,=\,& m_{1,1} e_1 + m_{1,2} e_2 + \ldots + m_{1, n} e_n\nonumber\\
g_2\, =\,& m_{2,1} e_1 + m_{2,2} e_2 + \ldots + m_{2, n-1} e_{n-1}\nonumber\\
g_3 \,=\, & m_{3,1} e_1 + m_{3,2} e_2 + \ldots + m_{3, n-2} e_{n-2}\nonumber\\
\cdots & \cdots\, \cdots\nonumber\\
g_{n-1} \,=\, & m_{n-1, 1} e_1 + m_{n-1, 2} e_2\nonumber\\
g_n \,=\,& m_{n , 1} e_1\nonumber
\end{align}
with $m_{i, j}\,\in\, {\C}^*\,=\, \C\setminus\{0\}$
so that $E_3^i \,=\, \langle g_1,\, \cdots ,\, g_i\rangle$. Write this as 
$$ g \,=\, M e,$$ for a matrix $M$. The total dimension of the space of $a, \, M$ is then 
$$d \,:=\, n+ (n+1)n/2\,=\, \frac{(n+2)(n+1)}{2}-1.$$

The actions of the diagonal torus $(\C^*)^{n}$ of ${\rm GL}(n, \C)$ acting on the framings associated to the points $p_1,\, p_2,
\, p_3$ is given by its natural action on the bases $e_i,\, f_i,\, g_i$ respectively, and so on the coefficients $a_i,\, m_{i,j}$.
This extends to a free action of $$(\C^*)^d \,\,=\,\, (\C^*)^{\frac{(n+2)(n+1)}{2}-1}$$ on $U$, by adding  $\C^*$ actions, acting individually on
each $m_{i,j},\, j\,\leq\, n-i$, or on the $a_i$.

\subsection{A birational map to projective space}

In the previous formulation, it is unclear how to extend the action globally, as it depends on a normalization. We can give a different 
encoding of our data (equivalent to the preceding one over $U$) which will give a birational map to projective space on which 
$(\C^*)^d$ will act, albeit with different weights. Let
$$\widetilde{\SMG}\, \,=\,\, \{(E,\,\beta_i^j,\, V)\,\,\big\vert\,\, (E,\, \beta_i^j)\,\in\, \SMG,\ V
\,\subset\, H^0(X,\, E),\ \dim V \,=\, n\}.$$
Then $\widetilde \SMG$ maps bijectively to $\SMG$ on the generic locus where $E$ is trivial, and even on the next-to-generic locus where the 
splitting type of $E$ is minimally non-trivial, i.e., the degrees of the line bundle components are $(1,\,0,\,\cdots ,\,0,\,-1)$.

We have the evaluation maps $V\,\longrightarrow\, E_{p_j}$, and so on exterior powers, which we compose with the maps $\beta^i_j$, to give us 
$$\widetilde{\beta}^i_j \,=\, \widetilde{\beta}^i_j(E,\,\beta^i_j,\, V)\,\in\, \bigwedge\nolimits^i V^* .$$
We then consider for each $j_1,\, j_2, \,j_3\,\in\, \{0,\,\cdots ,\,n\}$ summing to $n$, the elements 
$$\widetilde{\beta}^{j_1,j_2, j_3} \,=\, \widetilde{\beta}_1^{j_1}\wedge\beta_2^{j_2}\wedge\beta_3^{j_3}
\,\in\, \bigwedge\nolimits^n V^* .$$
A change of basis of $V$ changes these by the same scalar for all indices $j_1,\, j_2, \,j_3$; we thus have a well defined
element $\phi(E,\,\beta_i^j,\, V)$ of 
$\bP^d$, as long as one of the $\widetilde{\beta}^{j_1,j_2, j_3}$ is non-zero. This is not always the case: for example we can have 
a case when the bundle $E$ is non-trivial (so that $\widetilde \beta^{n,0,0} \,=\,\widetilde \beta^{0,n,0}
\,=\, \widetilde \beta^{0,0,n} \,=\,0$) and none of the other indices $j_1,\, j_2,\, j_3$ for which the corresponding $\widetilde \beta_i^{j_i}$ are non-zero actually sum to $n$. 
We thus must take the closure of a graph:
$$\widehat{\SMG} \,\,=\,\, \overline{\{((E,\,\beta_i^j, \,V), \,(\widetilde \beta^{j_1,j_2,j_3})
\,\in\, \SMG\times \bP^d\,\big\vert\,\,\phi(E,\,\beta_i^j,\, V) \,= \, (\widetilde \beta^{j_1,j_2,j_3})\}}.$$
This gives us a sequence of maps, which are birational:
$$\SMG\,\longleftarrow\,\widetilde\SMG\,\longleftarrow \,\widehat{\SMG}\,\longrightarrow\, \bP^d.$$

Under these correspondences, the complement of the coordinate hyperplanes in $\bP^d$ maps bijectively to our open subset $U$. There is a 
natural action of $(\C^*)^d$ (the maximal torus of ${\rm PGL}(d+1, {\mathbb C})$) on $\bP^d$; it is the complexification of an $(S^1)^d$ 
action on $\bP^d$ which is Hamiltonian, with moment map taking values in the standard simplex. Via our birational maps, the action 
transfers over to the space $U$, in a way equivalent to what we have already defined.  We would hope to extend this to all of $\SMG$; 
however this cannot be done without deforming $\SMG$.
 
\section{Okounkov bodies and toric degeneration}

To see how this can be done, we recall the theory of Okounkov bodies \cite{KK, EK} and the theorems of Andersen, \cite{An}, and Harada-Kaveh 
\cite{HK}. A beautiful introduction to the theory can be found in the Notices survey article of Escobar and Kaveh \cite{EK}. The 
Okounkov body of a variety $X$ depends on two pieces of data. The first is the choice of a space $S$ of sections of a very ample line 
bundle bundle $L$, giving an embedding into projective space; as such the elements of $S$ generate the homogeneous coordinate ring of 
$X$. Typically, $S$ is simply $H^0(X,\, L)$. The second is some coordinate system near a point $p$, given by an ordered set of functions 
$f_1,\,\cdots, \, f_d$ near the point. This can be thought of as corresponding to $d$ divisors with normal crossing at the point; as we will be 
thinking of them in terms of a lexicographic order, they can be thought of as a flag (indeed, it suffices to consider the successive 
intersections $D_1,\, D_1\cap D_2,\,\cdots$). Given any other function $g$ near $p$, we can expand $g$ into a power series $\sum_{\alpha = 
(\alpha_1,\cdots ,\alpha_d)}c_\alpha f_i^{\alpha_1}\cdots f_d^{\alpha_d}$. We can then define a $\Z^d$-valued valuation $\nu(g)$ of $g$ by
$$\nu(g) \,=\, \min\{\alpha\,\,\vert\,\, c_\alpha\,\neq\, 0\}$$
here the minimum is with respect to a lexicographical order. (There are more general valuations which could be used,
however this is what we need.)

Now consider the homogeneous coordinate ring $$\C[X]\,=\, \bigoplus_{m\in\Z_{\geq 0}}\C[X]_m$$ generated by products of elements of $S$ 
(we fix a trivialization of $L$ near $p$); $\C[X]_m$ corresponds to the $m$-fold product. We have a semi-group $F_S\,\subset\,
\Z_{\geq 0} \times \Z^d$
$$F_S \,=\, \bigcup_{m\geq 0} \{ (m, \,\nu(g))\,\,\vert\,\, g\,\in \, \C[X]_m\}.$$
Let $C$ be the closure of the convex hull of $F_S\cup\{0\}$; define the Okounkov body $\Delta_S$ as the intersection of $C$ with the 
plane $m\,=\,1$. A standard example would be given by $\bP^n(\C)$, with the $S$ being the sections of $\SO_{\bP^n(\C)}(1)$, and the standard 
coordinate planes defining the valuation; we obtain the standard simplex as the Okounkov body. This is of course the moment polytope for 
the standard torus action of $\bP^n$, and the same result holds in general for toric varieties.
 
There is a natural question: if the Okounkov body of a variety $X$ is the polytope of a toric variety, can the variety be deformed to a 
toric one? This imposes constraints on the body, since polytopes have a finite number of sides, and the polytopes here are rational. 
When the body is of this form, the answer, by a theorem of Anderson, is yes. For a variety $X$, let $X(S)$ denote the variety defined 
as the closure of the image of $X$ in $\bP(S)$. Anderson produces a flat degeneration $\X$ of $X(S)$, a family $\X\,\longrightarrow\,\C$ with 
fibers $X_t \,=\, X(S)$ for $t\,\neq\, 0$, and $X_0$ toric, with constant Okounkov body.

\begin{theorem}[{\cite[pp.~1184--1185, Theorem 1]{An}}] Let $\nu$ be the valuation associated to a flag of subvarieties of $X$, and
let $S \,\subset\, H^0(X,\,L)$ be a linear system such that $F_S$ is finitely generated, and the Okounkov body $\Delta_S$ is a rational polytope.
\begin{itemize}
\item The variety $X(S)$ admits a flat degeneration $\SSX$ to the (not necessarily normal) toric variety $X_0 $. The normalization of 
$X_0$ is the (normal) toric variety corresponding to the polytope $\Delta_S$.

\item If a torus $T$ acts on X , such that $S$ is a $T$-invariant linear system and the flag consists of $T$-invariant subvarieties, 
then the degeneration is $T$-equivariant.

\item Suppose $S'\,\subset \,S$ is a subsystem inducing a birational morphism $$\phi\,:\, X(S)\,\longrightarrow\, X(S'),$$ and whose 
semi-group is also finitely generated. The corresponding degenerations of $X(S)$ and $X(S')$ are compatible: there is a commuting diagram 
of flat families
$$ \begin{matrix}\SSX && {\buildrel{\Phi}\over{\longrightarrow}}&&\SSX'\\&\searrow&&\swarrow\\ &&\C&&\end{matrix}$$
\end{itemize}
\end{theorem}

The proof proceeds by a degeneration of the ring structure, transforming a ring with a filtration into a graded one. Once we have the 
toric degeneration we can ask what the corresponding Hamiltonian picture would be here. The answer is given by Harada and Kaveh, under 
the hypotheses that the family $\SSX\,\longrightarrow\, \C$ has an expression as a restriction of a uniform projection
$\bP^N\times\C\,\longrightarrow\, 
\C$, that the toric action on $X_0$ is the restriction of an action on $\bP^N$, that the K\"ahler forms on $\SSX$ are restrictions of 
the form on $\bP$, and invariant under the torus action, and finally that $\SSX$ is smooth away from the origin:

\begin{theorem}[{\cite[Theorems 1 and 2]{HK}}]\label{th-hk}
Let $X$ be a smooth $n$-dimensional projective variety, and let $\omega$ be a K\"ahler structure on $X$. Suppose that there exists a toric 
degeneration $\pi \,:\, \SSX \,\longrightarrow\, \C$ of $X$, satisfying the above conditions. Then:
\begin{itemize}
\item There exists a surjective continuous map $\phi\,:\, X\,\longrightarrow\, X_0$ which is a symplectomorphism restricted
to a dense open subset $U \,\subset\, X$.

\item There exists a completely integrable system $\mu \,=\, H_1,\,\cdots ,\,H_n$ on $(X,\,\omega)$, in the sense of
the Hamiltonians being continuous on $X$, smooth and defining a standard integrable system on $U$, such that its
moment image $\Delta$ coincides with the moment image of $(X_0,\, \omega_0)$, which is a polytope.

\item On the open dense subset $U$ of $X$, the integrable system $\mu \,=\, H_1,\,\cdots ,\,H_n$ generates a Hamiltonian torus action, 
and the inverse image $\mu^{-1}(\Delta^0)$ of the interior of $\Delta$ under the moment map $\mu$ of the integrable system lies in the 
open subset $U$.

\item Under the hypothesis of the semi-group $F_S$ being finitely generated, $\Delta$ is the
Okounkov body of $X$ under the linear system $S$ giving the projective embedding.
\end{itemize}
\end{theorem}

All the hypotheses of Theorem \ref{th-hk} seem quite likely to hold for the case which concerns us, apart from smoothness.

\section{Strata, and birational equivalence}

In order to apply Anderson's theorem to the case which interests us, we would like to be able to see how the property of having a 
rational finite Okounkov body gets transported through the birational equivalences given above. There is no question of the body being 
preserved: equivariantly birational toric varieties do not have the same polytope. To see what is at stake here, consider the simple 
case of $\bP^2$, with its standard torus action, standard affine coordinates and so on. We consider the identity projective embedding, 
given by section of $\SO_{\bP^2}(1)$. In a standard trivialization, its sections are $1,\,x,\,y$; those of $\SO_{\bP^2}(m)$ are
$1,\, x,\, y,\, x^2,\, xy,\, y^2,\, \cdots ,\, xy^{m-1},\,y^m$. The Okounkov body, corresponding to the divisors $x\,=\,0$
and $y\,=\,0$ is a simplex. Now blow up the origin $p$, 
$$\pi\,:\,\widetilde{\bP}^2 \,\longrightarrow\, \bP^2,$$ with exceptional divisor $E$ sitting over the origin. We have
coordinates $(\widetilde{x},\, 
\widetilde{y})$ on the blowup, with $(x,\,y)\, =\, (\widetilde{x},\, \widetilde{x}\widetilde{y})$, and we take the divisors cut out
by $(\widetilde{x},\, \widetilde{y})$ to define our Okounkov body. An embedding of $\widetilde{\bP}^2$ is given by sections
of $\pi^*({\SO}_{\bP^2}(2))(-E)$, spanned 
by $\widetilde{x},\, \widetilde{x}\widetilde{y},\, \widetilde{x}^2,\, \widetilde{x}^2\widetilde{y},\,
\widetilde{x}^2\widetilde{y}^2$, or, on $\bP^2$ 
the space of sections by $x,\, y,\, x^2,\, xy,\, y^2$, i.e., the space of sections of ${\SO}_{\bP^2}(2)\otimes \SI_p$, where $\SI$ is the ideal sheaf of 
the origin. These sections and their powers gives an Okounkov body which is a truncated simplex, obtained from the standard simplex by 
lopping off a corner. The result is indeed the moment polytope for $\widetilde \bP^2$.
 
It would be tempting to assert that the property of having a finitely generated rational polytope as Okounkov body survives any 
birational equivalence; we do not see any easy way of carrying this through. In our case, fortunately, we are assisted by the existence of a 
finite stratification of our spaces, with equivalences on the strata being produced by group actions which give a certain homogeneity to 
the space and reduce the number of added constraints to the Okounkov polytope to a finite number. Indeed, our space of bundles is a 
space of sheaves $E$, with additional structures given by the elements $\beta_i^j$. The types of possible sheaves $E$ are
$$\SO(m_1)^{\ell_1} \oplus \SO(m_2)^{\ell_2}\oplus \cdots \oplus \SO(m_s)^{\ell_s}\oplus (( \C_{p_1})^{\oplus k_1})\oplus ( ( \C_{p_2})^{\oplus k_2})\oplus ( ( \C_{p_3})^{\oplus k_3})$$
with the possibilities for $m_i, k_j$ being bounded by a degree constraint (summing to zero) and by stability. Stratifying the space $W$ 
used to build the moduli space by these types, the moduli space construction is basically the same in a neighborhood of any point in a 
stratum, one point mapping to any other by the action of the group we quotient by. Going further, and choosing an n-dimensional space 
$V$ of sections, there is a further stratification by the intersection of $V$ with the space of sections of different types 
$\SO(m_j)^{\ell_j},\,( \C_{p_i})^{\oplus k_i}$ and by the dimensions of the pullbacks of $E_{p_i}$ to $V$, and by the intersection 
patterns of all these spaces. Again, any neighborhood in the moduli construction of one pair $(E,\, V)$ in a stratum is equivalent to any 
other; what distinguishes them is the values of the coordinates $\widetilde{\beta}^{j_1,j_2, j_3}\,=\, \widetilde{\beta}_1^{j_1} \wedge 
\widetilde{\beta}_2^{j_2} \wedge \widetilde{\beta}_3^{j_3}$. Now consider for each point $p$ of a stratum the set $I_p$ of indices $ 
j_1,\,j_2, \,j_3$ for which $\widetilde{\beta}^{j_1,j_2, j_3}$ vanishes. It is easy to see that if $I_p,\, I_q$ are the same for $p,\, q$ in the 
same stratum, they are equivalent by actions of ${\rm GL}(n, \C)$ on $E_{p_i}$, $i \,=\, 1,\,2,\,3$ similar to those used in
Section \ref{se2.2}; basically we normalize two of the three flags, and the coordinates of the third give the different values.

We now consider the variety $\widehat{\SMG}$ mapping to $\bP^d$, with an isomorphism from the set $U$, to the complement of the 
coordinate hyperplanes in $\bP^d$. The variety $\widehat{\SMG}$ can thus be thought of as sitting inside a blow-up of various coordinate 
k-planes in $U$; the homogeneity argument of the previous paragraph tells us that the blowing up needed is indeed constant on the set of 
a stratum with constant $I_p$.

The variety $\widehat{\SMG}$ is thus embedded by the linear system $W$ of sections of a line bundle $L$, given as the pull-back of a 
line bundle $\SO(m)$ on $\bP^d$, twisted by a divisor $-E$; it corresponds on $\bP^d$ to $H^0(\bP^d,\, \SO(m)\otimes \SI) $ for some ideal 
corresponding to a family of coordinate $k$-planes, possibly with multiplicity, but with a finite number of generators corresponding to 
$\widetilde \beta^{j_1,j_2, j_3}$ vanishing to an appropriate order on our strata. We choose as origin the point $p_0$ corresponding to 
the trivial bundle with all $\beta_i^j$, $j\,=\, 1,\, \cdots ,\, n-1$, zero (i.e., no flags) and $\beta_i^n$ given by the
${\rm SL}(m, \C)$-structure, blowing 
up further if necessary to ensure that the point and its images under rescaling the $\beta_i^n$ are smooth; we pick a system of divisors 
$D_i$ in the preimage of the coordinate planes on $\bP^d$, and again there are a finite number of rational constraints on the order of 
the $\widetilde{\beta}^{j_1,j_2, j_3}$. In short, we have a rational subpolytope of the simplex, and we can apply Anderson's theorem to 
$\widehat{\SMG}$.
 
We now want to project from $\widehat{\SMG}$ to ${\SMG}$, again desingularising the image in $\SMG$ of the point $p_0$ and its 
re-scalings; let the projection be called $\rho$. It is of course an isomorphism on $\SMG$ away from a set of codimension at least two, 
and in particular in a neighborhood of the base point, consequently mapping our coordinate planes $D_i$ down to $\SMG$. Now take the 
direct image under $\rho$ of $L$. As the sections of $L$ are all products of length $m$ of the $\widetilde \beta^{j_1,j_2, j_3}$, the 
direct image takes values in the $m$-th power $L_0^m$ of the natural determinant line bundle $L_0$ over the moduli space. Pushing down 
and pulling back the sections in $W$, we obtain a subspace $W' \,=\, \rho^*\rho_*W \,\subset\, W$ of the space $W$, and we can think of $W'$ as 
a subspace of $H^0(\SMG,\, L_0)$. Again, by the homogeneity of our stratification, its Okounkov body is cut out in that of 
$\widehat{\SMG}$ by a finite number of constraints, and so is a rational polytope.

In short, again by Anderson's theorem:

\begin{theorem}
The space ${\SMG}$ admits a degeneration to a toric $\SMG^0$. The torus actions on the framings on $\SMG$ extend to torus actions on 
$\SMG^0$.
\end{theorem}

\section{Hamiltonians and imploded cross-sections}

\subsection{Doubles and the construction of the moduli space}

As we noted in the introduction, our space of framed sheaves admits a symplectic interpretation, as a space of representations of the 
fundamental group of the three-punctured sphere into ${\rm SU}(n)$, decorated with some extra structure at the marked points. We would like in 
this section to explore our ``almost toric structure'' from this symplectic point of view.

We first recall some constructions, made in \cite{AMM} and \cite[Section 6]{HJS1}, of the double $DK \,=\, K\times K$ of a compact 
semi-simple group $K$, and of its imploded cross-section. Setting $DK \,=\, \{(u,\,v)\,\in\, K\times K\}$, we can think of $DK$ as the space of 
representations of the fundamental group of a cylinder, i.e., flat $K$-connections, but on a bundle equipped with trivializations at two 
points $q_1,\, q_2$, one at each end of the cylinder. The element $u$ represents parallel transport from $q_1$ to $q_2$, and the element 
$v$ represents parallel transport around the end, starting at $q_2$. The space of representations comes equipped with a two-form, 
defined by:
$$
\omega\,\,=\,\,-\frac12\bigl(Ad(v)u^*\theta_L,\,u^*\theta_L\bigr)
-\frac12\bigl(u^*\theta_L,\,v^*(\theta_L+\theta_R)\bigr).
$$
Here $\theta_L$ is the left Maurer-Cartan form, and $\theta_R$ is the right form. The form $\omega$ is not closed, and $d\omega$ is
given by the pull-back under a suitable group-valued moment map of the canonical three-form on the compact group. In our case, the group action is 
one of $K\times K$ on $DK$:
$$
(k_1,\,k_2)(u,\,v) \,= \,(k_1uk_2^{-1},\, k_2 vk_2^{-1}).
$$
In terms of flat connections on the cylinder, this is the action corresponding to changes of the trivializations at $p_1, \, p_2$. The 
moment map here is given by $(Ad(u)v^{-1},\,v)\,\in\, K\times K$, and $DK$ is a quasi-Hamiltonian $K\times K$ space. This three-form 
disappears under restriction to a torus or under quasi-Hamiltonian reduction, and as this is what we will be doing, the two form will 
end up being closed.

The next step will be to restrict $v$ to lie in the fundamental alcove $\Delta$ of the torus $T$, the polytope representing conjugacy 
classes in $T$. For ${\rm SU}(n)$, we take a logarithm and divide by $2\pi\sqrt{-1}$, so that
$$ v \,=\, \exp (2\pi\sqrt{-1}(\alpha_1, \,\alpha_2,\,\cdots ,\,\alpha_n)).$$
The alcove $\Delta$ is precisely the set of ordered eigenvalues 
$$\Delta \,=\, \left\{(\alpha_1\,\geq\,\alpha_2\,\geq\,\cdots\, \geq\,\alpha_n\,\geq\, \alpha_1-1)\,\,\big\vert\,\,
\sum_{i=1}^n \alpha_i \,=\, 0\right\}$$
for each marked point. This simplex of weights decomposes into faces $$\Delta^I,\ \ \, I \,=\, \{i_1\,<\,i_2\,<\,\cdots\,<\,i_k\}$$
consisting of the weights for which $\alpha_{i}\,>\,\alpha_{i+1}$ if and only if $i\,\in\, I$. There is a similar decomposition for arbitrary $K$. 

We now can restrict to $K\times \Delta\,\subset \,DK$. For $v$ in the interior of $\Delta$, the form is non-degenerate; but on the
boundary strata $\Delta_I$ of $\Delta$ there are null directions, which are orbits of the commutator group $[K_I,\, K_I]$
of the stabilizer $K_I$ of the elements of $\Delta_I$; see \cite{HJS1}. We quotient these out, to obtain the imploded cross-section:
\begin{equation}\label{imploded}
DK_{\mathrm{impl}}\,\, =\,\, \sqcup_I K_I/[K_I, K_I] \times \Delta_I .
\end{equation}
This is a stratified quasi-Hamiltonian $K\times T$--space. The action at $p_2$ is now restricted to the torus $T$.
 
We can use the $DK_{\mathrm{impl}}$ in \eqref{imploded} as building blocks for the construction of the moduli spaces $M$
of framed representations on the trinion. Indeed, take three copies of 
$DK_{\mathrm{impl}}$:
\begin{equation}
\left\{ (u_1,\,v_1,\, u_2,\, v_2,\, u_3,\,v_3) \,\in\, DK_{\mathrm{impl}}^3\right\}.
\end{equation}
Together, their elements represent holonomies $v_i\,\in\, \Delta$ at the punctures $p_i$,\, $i\,=\,1,\,2,\,3$, along with
parallel transport $u_i$ to a common point $p_0$. The constraint that they define together a representation of the fundamental group 
$$
u_1v_1u_1^{-1} u_2v_2u_2^{-1} u_3v_3u_3^{-1}\,=\,1
$$
is equivalent to setting to $1$ the moment map for the diagonal action of $K$, with a quotienting by $K$ eliminating the choice
of an arbitrary trivialization at $p_0$. Thus, the end result is:
 
\begin{theorem}[{\cite{HJS1}}]\label{JHS1}
The space $M$
$$
M\,= \, \left\{ (u_1,\,v_1,\, u_2,\, v_2,\, u_3,\,v_3) \,\in\, DK_{\mathrm{impl}} ^3\,\,\big\vert\,\,
u_1v_1u_1^{-1} u_2v_2u_2^{-1} u_3v_3u_3^{-1}\,=\,1\right\}/K
$$
is a Hamiltonian $T^3$ space with moment map 
$$
(u_1,\,v_1, \,u_2, \,v_2,\, u_3,\,v_3) \,\longmapsto\, (v_1,\, v_2,\, v_3)\,=\,V.
$$
Its Hamiltonian reduction at $V=(v_1,\, v_2,\, v_3)$ is the moduli space $M_V$ of parabolic bundles with holonomy
$v_i$ at $p_i$. The reduction fixes the $v_i\,\in\, \Delta$, and quotients out the orbit of $K_i/[K_I,\, K_I]$, eliminating the framings of the subquotients at the punctures.
\end{theorem}

\subsection{Hamiltonian flows and vector fields}
 
We would like the  Hamiltonian flows on the generic locus of $M$, corresponding to the toric structure on the degeneration. This locus will be given by first asking that the $v_i$ lie in the interior $\Delta_0$ of $\Delta$ (over which $DK_{\mathrm{impl}}$ is simply $K\times \Delta_0$), and also asking that the frames represented by $u_i$ satisfy certain transversality conditions. As this is simpler to see in the complex domain, we first complexify $M$ and $DK_{\mathrm{impl}}$ to  spaces $M_\C , (DK_{\mathrm{impl}})_C$ and consider everything over the complex domain, so that the $u_i$ belong to the complexified group and the $v_i$ to the complement of the root planes in $T_\C$. 
 
We specialize to ${\rm SU}(n)$ and its complexification ${\rm SL}(n,\C)$. The interior $\Delta_0$ of the fundamental alcove corresponds to distinct 
eigenvalues and in the complex picture (for $M$) to full flags. With the usual conventions, these flags correspond in the torus trivialization to 
$\langle e_1\rangle\,\subset\, \langle e_1,\,e_2\rangle\,\subset\, \cdots$. We will keep this ordering for $p_1$, and take the inverse
order $\langle e_n\rangle\,\subset\, \langle e_n,\, e_{n-1}\rangle\,\subset\, \cdots$ for $p_2,\, p_3$. Matrices denoted by $b$ will be triangular, either upper ($b^+$) or lower ($b^-$); those denote by $n$ will be strictly triangular (i.e., with 1's on the diagonal) and either upper ($n^+$) or lower ($n^-$); those denoted by $d$ will be diagonal. Corresponding Lie algebra elements will be denoted by $\beta^\pm, \mu^\pm, \delta$, with $\beta^\pm= \mu^\pm + \delta$; variations of the $v$ will denoted by $\zeta$.
 
Over a generic locus, then, we can factor the $u_i$ as $b_i^+n_i^-= n_i^+d_in_i^-$, or the other way round. We can also suppose  that the flag transported from $p_1$ and the flag transported from $p_2$ intersect transversally at 
$p_0$, and so we can normalize our frame at $p_0$ by choosing the intersections as a basis so that $u_1$ is an upper triangular matrix 
$b_1^+= n_1^+$ (i.e., preserving the flag at $p_1$), $u_2$ is a strictly lower triangular matrix $n_2^-$, (preserving the flag and the trivialization from $p_2$), and, 
with some further genericity, $u_3$ factors uniquely as $b_3^+n_3^-$, a product of an upper triangular matrix and a strictly
lower triangular matrix, which we also write as $n_3^+d_3n_3^-$.

We then have the relation
\begin{equation}\label{relation}b_1^+v_1 (b_1^+)^{-1} n_2^-v_2(n_2^-)^{-1} \,=\, b_3^+n_3^-v_3^{-1} (n_3^-)^{-1}(b_3^+)^{-1}.\end{equation}
In this normalisation, $n_3^+$ represents the relative positions of the flags, and the diagonal components $d_1,d_3$ the framings.   As we are on the generic, stable, locus, we can vary these (framed) flags (i.e., $b_3^+$) freely while keeping the $v_i$ fixed, and there
will be a corresponding variation of the $u_i$ on the unitary locus; this, however, will be accompanied by a corresponding variation of $n_3^-, b_1^+, n_2^-$ in order to preserve \ref{relation}.
 
We now evaluate the  form $\omega$ on two-parameter infinitesimal variations
$u, \,v $ over $(K\times \Delta_0)_\C$. We set 
\begin{align}
& b^+ (1+\epsilon \beta^++\delta \widehat\beta^+) n^-(1+\epsilon \mu^-+\delta \widehat\mu^-)\nonumber \\
&=\, b^+ n^-(1+ \epsilon ((n^-)^{-1}\beta^+n^- +\mu^-)+\delta ((n^-)^{-1}\widehat\beta^+n^- +\widehat\mu^-))
\nonumber
\end{align}
and 
$$
v (1+\epsilon \zeta +\delta \widehat \zeta),
$$
and so the form applied to the variations becomes
\begin{align}&\omega((\beta^+,\,\mu^-,\, \zeta),\, (\widehat\beta^+,\,\widehat\mu^-,\,\widehat \zeta))\, =\nonumber\\
\frac 12 tr( (v(n^-)^{-1}\beta^+n^-v^{-1})& ((n^-)^{-1}\widehat\beta^+n^-)- (v(n^-)^{-1}\widehat\beta^+
n^-v^{-1}) ((n^-)^{-1}\beta^+n^-))\nonumber\\
+ \frac 12 tr( (v(n^-)^{-1}\beta^+n^-v^{-1})& \widehat\mu -
(v(n^-)^{-1}\widehat\beta^+n^-v^{-1})\mu^-)\nonumber\\
+ \frac{1}{2\pi i}tr ((n^-)^{-1}\beta^+n^- +\mu^-)\widehat \zeta &- ((n^-)^{-1}\widehat\beta^+n^- +\widehat\mu^-)\zeta).
\nonumber 
\end{align}

For $(\beta^+,\,\mu^-, \,\zeta) \,= \,(0,\, \mu^-,\,0)$ this becomes:
\begin{align}
\omega((0,\mu^-, 0), (\widehat\beta^+,\widehat\mu^-,\widehat \zeta)) \,=\, &
\frac 12 tr(- (v(n^-)^{-1}\widehat\beta^+n^-v^{-1})\mu^-)\nonumber\\
 =\, & \frac 12 tr(-\widehat\beta^+(n^-v^{-1}\mu^-v(n^-)^{-1}))\nonumber\\
  =\, & \frac 12 tr(-\widehat\mu^+(n^-v^{-1}\mu^-v(n^-)^{-1}))\label{pairing}
\end{align} 
 
Note that this implies that the various vectors tangent to the space $n^+B^{-}$, that is $\mu^-,\, \delta$ pair amongst themselves to 
zero.   
The formula (\ref{pairing}) also shows that the basis tangent vector $(\mu(i,\,j)^+)_{k,l}\,=\, \delta_{k,i}\delta_{l,j}$,\, $i\,
<\,j$, pairs non-trivially with $(\mu(j,\,i)^-)_{k,l}\,=\, \delta_{k,j}\delta_{l,i}$
with coefficient $(b^-)_{j,j}(b^-)_{i,i}^{-1}v_iv_j^{-1}$. Furthermore, again for $i\,<\,j$, $\mu(i,\,j)^+$ only pairs non-trivially with 
 $\mu(i',\,j')^-$ if $j'\,\leq\, i\,<\,j\,\leq\, i'$. This ensures that, for $i\,<\,j$ the symplectic form gives a
nondegenerate pairing between the $\mu(i,\,j)^+$ and $\mu(j',\,i')^-$, which we write as 
$$ \omega (\mu(i,\,j)^+, \mu(i',j')^-)\, =\, A_{i,j;i'j'}$$
so that on linear combinations $P\,=\, \sum_{i<j}p_{i,j}\mu(i,\,j)^+$, $Q\,=\, \sum_{i>j}q_{i,j}\mu(i,j\,)^-$, we have, schematically
$$ \omega (P,\,Q) \,=\, p^TAq.$$

\subsection{Hamiltonians}

The formula above also suggest that the components of the group $N^+$  of strictly upper triangular 
matrices should be the Hamiltonians generating the vector fields tangent to $b^+N^{-}$. Now note that the group $N^+$ admits a 
logarithm (inverting the exponential map) with values in the Lie algebra $\gn^+$ of strictly upper triangular matrices. Let $\xi\,\in\,\gn^+$ 
with $\exp(\xi) \,=\, n^+$, so that $\xi\,=\, \log(n^+)$. We have the formula
\begin{align}\label{expsum}\frac{d}{dt} \exp(\xi+t\zeta) |_{t=0} \,=&\, \exp (\xi) \cdot \Big( \frac{1- \exp(-ad(\xi))}{ad(\xi)}\Big) \zeta\\
=&\, \exp(\xi)(1-\frac{ad(\xi)}{2} + \frac{ad(\xi)^2}{6}-\ldots ) \zeta.\nonumber\end{align}
Since the adjoint action is nilpotent, this gives a bijective map between the tangent vectors to $\gn^+$ and the tangent vectors to
$N^+$. Consider the components of $\xi$, for $i\,<\,j$:
$$
H_{i,j}(  n ^+,b ^-, v) \,=\,\log(n ^+)_{i,j}\,=\, \xi_{i,j}.
$$
By the formula (\ref{expsum}) above, the differential of $H_{i,j}$ evaluated on the left invariant vector fields $P\,=\,
\sum_{i<j}p_{i,j}\beta(i,j) ^+$ is a combination of coefficients:
$$dH_{i,j} \,=\, p_{i,j} + \sum_{i'<i, j'>j} C_{i',i; j',j}(b^+) p_{i',j'}$$
or, schematically,
$$dH_{i,j} \,=\, (p^T C)_{i,j}.$$
Combining, since we want for our Hamiltonian vector fields $X_{H_{i,j}} \,=\, \sum_{i'j'} M_{i',j'; i,j}\mu(j',i')^-$ that, for all $P$
$$dH_{i,j}(P) \,=\,-\omega(P, X_{H_{i,j})}$$
we have 
$$ (p^T C)_{i,j} \,=\, -(p^TAM)_{i,j}$$
giving $M\,=\, A^{-1}C$ and
$$X_{H_{i,j}} \,=\, -\sum_{i'j' }(A^{-1}C )_{i',j'; i,j}\mu(j',i')^-$$
for $i\,<\,j$.

Finally, we have two "symplectically dual" quantities. First, the Hamiltonians associated to the torus coefficients $v$: set as Hamiltonians 
$$H^V_{l}\, =\, \frac {1}{2\pi i} \log (v )_l,\ \ \,  l \,=\, 1,\,\cdots, \,n-1.$$
In the complex domain these would be multi-valued. They give the torus actions associated to $M$ by Theorem \ref{JHS1}, acting on
the $u_i$ by right multiplication. Their vector fields are given by left invariant vector fields $\delta$ corresponding to the diagonal action on the right.  By \eqref{pairing}, they
Poisson commute with the $H_{i,j}$.  Dually, we can also  take as Hamiltonians the entries of the  diagonal matrix $d$ in a decomposition $n^+dn^-$; again they Poisson commute with the $H_{i,j}$. Their vector fields will be vectors $\zeta$ tangent to the diagonal matrices $v$.

Summarizing:

\begin{theorem}
The Hamiltonians $H_{i,j},\, 1\,\leq\, i\,<\,j\,\leq\, n$ and $H^V_{l}$,\,  $l\,=\,1,\, \cdots ,\, n-1$ generate a
completely integrable Hamiltonian system (i.e., a holomorphic Lagrangian foliation) in the generic set of $(DK_{\mathrm{impl}})_\C$, with leaves $n^+B^-\times\{v\}$. Similarly, the Hamiltonians $H_{i,j},\, 1\,\leq\, i\,<\,j\,\leq\, n$ and the components $d_i$,\,  $l\,=\,1,\, \cdots ,\, n$ of $d$ generate a
completely integrable Hamiltonian system (i.e., a holomorphic Lagrangian foliation) in the generic set of $(DK_{\mathrm{impl}})_\C$, with leaves $n^+N^-\times\ (\Delta_0)_\C$ \end{theorem}

The Hamiltonians $H_{i,j}$ are of course defined only on a generic locus; we give other Hamiltonians, more globally defined, which give the same 
ring, and which descend to $M$.

\subsection{Global and reducible Hamiltonians}

We now have Lagrangian foliations, with associated Hamiltonians, on each of the three $DK_{\mathrm{impl}}$ factors. We would like to combine these into one set. Thinking of $DK_{\mathrm{impl}}$ or its complexification as parametrizing  framed representations, or their complexification, we have at each puncture $p_i$ a torus-valued 
endomorphism $v_i$, and a flag of subspaces $$E_{i,j}\,=\,\langle e_1,\, e_2,\,\cdots ,\,e_j\rangle,\, \ \text{ for }\ i\,=\,1$$
and $$E_{i,j}\,=\, \langle e_n,\, e_{n-1},\,\cdots ,\, e_{n-j+1}\rangle,\, \ \text{ for }\ i\,=\, 2,\, 3.$$
The map $v_i$ generates an automorphism $ V_{i,j}\,=\, \bigwedge^j(v_i)$ (a scalar) of $\bigwedge^jE_{i,j}$, essentially as the product 
of the first $j$ eigenvalues. Now transport the basis $e_j$ of $E_{i,j}$ to $p_0$ using the $u_i$, obtaining vectors $u_i(e_j)$ in the 
fiber $E_{p_0}$ of the bundle $E_{p_0}$, and so, taking their product, a $j$-vector $u_i(E_{i,j})$ in $\bigwedge^j(E_{p_0})$. The fiber 
$E_{p_0}$ has a top $n$-form $dvol$ given by the ${\rm SL}(n,{\mathbb C})$ structure; for $j_1+j_2+j_3 \,=\, n$, consider the matrix
$$G_{j_1, j_2, j_3} \,=\,(u_1(e_1,\,\cdots ,\,e_j),\, u_2(e_n,\,\cdots ,\,e_{n-j_2+1}),\, u_3(e_n,\,\cdots ,\,e_{n-j_3+1})$$
and take the volume:
$$dvol(u_1(E_{1,j_1})\wedge u_2(E_{2,j_2})\wedge u_3(E_{3,j_3})) \,=\, \det G_{j_1, j_2, j_3}.$$
Note that decomposing $u_1$ as $b_1^-n_1^+$ we can replace $u_1$ by $b_1^-$; likewise, setting $u_j= b_j^+n_j^-, j= 2,3$, we can replace $u_j$ by  $b_j^+$ for $j=2,3$
This is invariant under the left action of ${\rm SL}(n,\C)$ on the $u_i$. It is not invariant under the action of the tori associated to the punctures.
 (On our generic set we could, for example,
 use the right torus action on $u_i$ to normalize its principal minors $\det (u_i)_{s,t}$,\, $s,\, t\,\in\, \{1,\,\cdots ,\,j\}$,\,
$i\,=\, 1,\,3$ and $\det (u_2)_{s,t}$,\, $s, \,t\,\in\, \{n-j+1,\, \cdots ,\,n\}$,\, $i\,=\, 1,\,3$ to one; but then you lose the  ${\rm SL}(n,\C)$-invariance)
We set for $j_1+j_2+j_3 \,=\, n$,\, $j_i \,\in\, \{0,\, \cdots ,\,n-1\}$,\,,
 \begin{equation}
\widetilde H^G_{j_1, j_2, j_3} \,=\, \det G_{j_1, j_2, j_3}.
\end{equation}
In our normalization of $u_1\,=\, b_1^+=\,n_1^+ d_1 ,\, u_2 \,= \,n_2^-,$ and $u_3 \,=\, b_3^+ n_3^-=\, n_3^+d_3 n_3^-=$,  the Hamiltonians become determinants of submatrices of $b_3^+$, times a product of entries of $d_1$:
$$H^G_{j_1, j_2, j_3} \,= \, d_{1,1}...d_{1,j_1}\det( (b_3^+)_{i,j},\, i\,=\, j_1+1,\,\cdots ,\,j_1+j_2,\, j \,=\, n-j_2+1,\, \cdots ,\,n).$$

\begin{proposition}
In the above normalization on the generic set there is a biholomorphism relating the $H^G_{j_1, j_2, j_3}$ and (the Hamiltonians
$H_{i,j}$ and the coefficients of $d_1, d_3$)  so that they represent a reparametrization of our Hamiltonians.
\end{proposition}

\begin{proof}
Note that in both cases there are $(n+4)(n-1)/2$ functions; now let us fix the $d_i$; this leaves $n(n-1)/2$ functions on both sides.. We first remark that under the exponential map, the $H_{i,j}$ map 
biholomorphically to the coefficients of $n_3^+$. The map giving $H^G_{j_1, j_2, j_3} $ from $(n_3^+)_{i,j}$ is given above; conversely, 
first remark that for $j_3\,=\, 1$, the Hamiltonians $H^G_{j_1, j_2, j_3} $ are simply (up to the diagonal elements of $b_1^+$, and a sign) 
the coefficients $(n_3^+)_{i,n}$. Now take $j_3 \,=\, 2$,\, $j_1\, =\, n-3$. The value of $(n_3^+)_{n-1, n-1}$ being one, we can compute 
$(n_3^+)_{n-2, n-1}$ from the determinant $H^G_{n-3, 1, 2} $; keeping $j_3\,=\, 2$, and decreasing $j_1$, we can then compute successively 
the coefficients $(n_3^+)_{i, n-1}$ from $H^G_{i-1, 1, 2} $. Now begin again with $j_3\,=\, 3$, we can again work our way up, computing 
successively $(n_3^+)_{n-3, n-2},\, (n_3^+)_{n-4, n-2},\, \cdots$ and so on. We then proceed in the same fashion for $j_3\,=
\, 4,\,\cdots ,\,n-1$.
\end{proof} 

The $H^G_{j_1, j_2, j_3} $, by the previous section, Poisson commute. They are invariant under the left action of ${\rm SL}(n,\C)$, and so descend to commuting Hamiltonians on the quotiented space $M$, by the results of \cite{AMM}, section 5.  Note that the $H^G_{j_1, j_2, j_3} $ give us well defined functions on ${\rm SL}(n,\C)\times T_\C$; however, when restricted to
${\rm SU}(n)\times\Delta$, they will not be invariant under $[K_I,\, K_I] $, except on the generic locus $\Delta_0$ where $[K_I,\, K_I] $ is trivial.

\subsection{Reality} 

We thus have an isotropic holomorphic foliation along (an open dense subset) of the complexification $M_\C$ of our space $M$, cut out by 
the functions $H^G_{j_1, j_2, j_3} $. This does not yet give us a real Lagrangian foliation on $M$, as indeed, even on $M$, this 
foliation is very far from being real. We must adjust.

The subset $M\,\subset\, M_\C$ is the fixed point locus of an involution $I$, induced from the involution on ${\rm SL}(n,\C)$ given by 
$$I(g) \,=\, (g^*)^{-1}.$$
Our symplectic form $\omega$ is invariant under $I$, in the sense that $\overline{I^*(\omega)} \,=\, \omega$. We will want as
Hamiltonians functions invariant under the involution $I$. 
For the torus Hamiltonians, this is straightforward: set 
$$ \alpha_k^l\,=\,Re(H^V_{k,l}).$$
These are the coefficients $\alpha $ in the fundamental alcove, but with the order reversed for $k\,=\,2,\,3$. For the other
Hamiltonians, we take
\begin{equation}
H^{GG}_{j_1, j_2, j_3} \,=\, A(\alpha_ 1^{j_1+1}-\alpha_ 1^{j_1})(\alpha_2^{n-j_2}-\alpha_ 2^{n-j_2+1})
(\alpha_3^{n-j_3}-\alpha_3^{n-j_3+1}) (I^*H^G_{j_1, j_2, j_3})^*H^G_{j_1, j_2, j_3};\end{equation}
here $A$ is the product $(\alpha_1^1-\alpha_1^n-1)(\alpha_2^1-\alpha_2^n-1)(\alpha_3^1-\alpha_3^n-1)$. The $\alpha$--factors
and $A$--factors are there to ensure that the Hamiltonian is invariant under the action of $[K_I,\, K_I]$, on the strata where this
is non-trivial. When we are on the unitary locus, the $H^{GG}$ are up to a real factor the determinants of a Hermitian matrix of the form
$$\begin{pmatrix} a_{1,1}&a_{1,2}&a_{1,3}\\ a_{1,2}^*&a_{2,,2}& a_{2,3}\\a_{1,3}^*&a_{2,3}^*&a_{3,3}\end{pmatrix},$$
with the $a_{i,i}$ diagonal and real. As the matrix is Hermitian, the determinant is real; it is invariant both under
the ${\rm SU}(n)$ action and the unitary torus actions. It is the product of $H^G_{j_1, j_2, j_3}$ and its complex conjugate, and so represents the norm squared of $H^G_{j_1, j_2, j_3}$, again up to some factors. 

There remains to see that these Hamiltonians commute. Denote $\{j_1,\, j_2,\, j_3\}$ by $\sigma$, and $\{k_1\,,k_2,\, k_3\}$ by
$\rho$. We already have $X_{H^G_\sigma} (H^G_\rho)\,=\,0$; also $H^G_\rho$ is constant along the $H^G_\sigma$ flow. If a
function is constant, its norm squared is too, and so $X_{H^G_\sigma} (H^{GG}_\rho)\,=\, 0$. Likewise,
$\overline{X_{H^G_\sigma}} (H^{GG}_\rho) \,=\, 0$, and so 
$$X_{H^{GG}_\sigma} (H^{GG}_\rho) \,=\, 0 .$$
The functions Poisson commute amongst themselves, and evidently commute with the torus Hamiltonians.

\section{Degenerations of bundle moduli}

Let us start with the moduli space $\SMS(Y_1)$ of ${\rm SL}(n,\C)$ bundles on a smooth genus $g$ curve $Y_1$, or equivalently, the space of 
representations of the fundamental group of $Y_1$ into ${\rm SU}(n)$. The results of \cite{BHu} give us a degeneration of the moduli spaces 
${\SMS}(Y_t)$ associated to a degeneration of $Y_1$ through smooth curves $Y_t$ into a ``totally nodal'' curve $Y_0$. This is a curve 
obtained by glueing $(2g-2)$ three-punctured spheres $X_i$ pairwise at their punctures, with $3g-3$ glueings in all. The limiting moduli
space ${\SMS}(Y_0)$ is a quotient of the product of the moduli spaces $\SMS(X_i)$ of our framed parabolic sheaves over trinions, where the quotient is by 
the anti-diagonal torus actions on the trivializations associated to the elements of $\SMS(X_i)$: that is, if the puncture $p_i$ of 
$X_i$ is being identified, to $p_j$ of $X_j$, we identify the trivializations associated to elements of $\SMS(X_i)$ at $p_i$ with the 
trivializations associated to elements of $\SMS(X_j)$ at $p_j$; note that the glueing is only partial, of subquotients of the bundle. We 
then quotient by the anti-diagonal torus action preserving these identifications. This quotienting can either be done symplectically, or 
holomorphically. Note that symplectically, there is a reversal of the $\alpha_i$ on the two  punctures being glued: $\alpha^i_1= -\alpha_2^{n-i+1}$, hence the anti-diagonal quotient instead of the diagonal one.

Thus we have a degeneration of ${\SMS}(Y_t)$ into a limit ${\SMS}(Y_0)$, which is a quotient, symplectic or holomorphic, of a product of 
$\SMS(X_i)$. Note that this is not a degeneration with a constant fiber away from $0$; the ${\SMS}(Y_t)$ are different, and indeed there 
is a Torelli theorem telling us that ${\SMS}(Y_t)$ determines $Y_t$ \cite{BGL}, \cite{GL}.

The ${\SMS}(X_i) $ themselves are not toric, except for ${\rm SL}(2,\C)$, when they are identified with $\bP^3(\C)$. However, since we now 
have a toric degeneration of the $\SMS(X_i)$ which is equivariant for the torus actions at the punctures, we now have one for 
${\SMS}(Y_0)$, and so a degeneration for ${\SMS}(Y_t)$.

\end{document}